\documentclass[11pt,reqno]{amsart}
\usepackage{amsmath, amsthm, amssymb, amsfonts}
\usepackage{bbm}
\usepackage{bm}
\usepackage{graphicx}
\usepackage{xcolor}
\usepackage{enumitem}
\usepackage{tikz}
\usepackage{tikz-cd}
\usepackage{pgfplots}
\usepackage{rotating}
\usepackage{hyperref}
\usepackage{cleveref}

\pgfplotsset{compat=1.18}

\theoremstyle{plain}

\usepackage[backend=biber,sorting=nty]{biblatex}

\numberwithin{equation}{section}

\newtheorem{theorem}{Theorem}[section]
\newtheorem{lemma}[theorem]{Lemma}
\newtheorem{proposition}[theorem]{Proposition}

\newtheorem*{claim}{Claim}

\theoremstyle{definition}
\newtheorem{definition}[theorem]{Definition}

\newtheorem{example}[theorem]{Example}

\newcommand{\func}{\operatorname}

\renewcommand{\c}[1]{\langle #1 \rangle}

\newcommand{\R}{\mathbb{R}}

\newcommand{\N}{\mathbb{N}}

\newcommand{\cal}{\mathcal}
\renewcommand{\frak}{\mathfrak}

\renewcommand{\bf}{\mathbf}

\newcommand{\supp}{\func{supp}}

\newcommand{\dom}{\func{dom}}
\newcommand{\ran}{\func{ran}}

\newcommand{\A}{\cal{A}}

\newcommand{\non}{\func{non}}

\newcommand{\M}{\mathcal{M}}

\newcommand{\FIN}{\bf{FIN}}

\let\strokeL\L
\DeclareRobustCommand{\L}{\ifmmode\mathbf{L}\else\strokeL\fi}

\newenvironment{midproof}[1][Proof]
  {\begin{proof}[#1]}
  {\end{proof}}

\DeclareFontShape{OT1}{cmr}{bx}{sc}{<-> cmbcsc10}{}

\title{Mad families of Gowers' infinite block sequences}
\author{Clement Yung}

\begin{document}
\begin{abstract}
    Call a subset of $\FIN_k$ \emph{small} if it does not contain a copy of $\c{A}$ for some infinite block sequence $A \in \FIN_k^{[\infty]}$. Gowers' $\FIN_k$ theorem asserts that the set of small subsets of $\FIN_k$ forms an ideal, so it is sensible to consider almost disjoint families of $\FIN_k$ with respect to the ideal of small subsets of $\FIN_k$. We shall show that $\frak{a}_{\FIN_k}$, the smallest possible cardinality of an infinite mad family of $\FIN_k$, is uncountable.
\end{abstract}

\maketitle

\section{Introduction}
For each $k \in \omega \setminus \{0\}$, let $\FIN_k$ denote the set of all mappings $x : \omega \to \{0,\dots,k\}$ such that $\supp(x)$ is finite and $k \in \ran(p)$. The \emph{support} of some $x \in \FIN_k$ is the set of all $n \in \dom(x)$ in which $x(n) \neq 0$. For two such maps $x,y$ of disjoint support, we may define a partial binary operation $+$ on $\FIN_k$ by stipulating that:
\begin{align*}
    (x + y)(i) = 
    \begin{cases}
        x(i), &\text{if $i \in \supp(p)$}, \\
        y(i), &\text{if $i \in \supp(q)$}, \\
        0, &\text{otherwise}. 
    \end{cases}
\end{align*}
Note that $(\FIN_k,+)$ forms a partial semigroup. Similar to the previous section, we may define a partial order $<$ on $\FIN_k$, where for any $x,y \in \FIN_k$:
\begin{align*}
    x < y \iff \max(\supp(x)) < \min(\supp(y)).
\end{align*}
A \emph{block sequence} of elements of $\FIN$ is thus a $<$-increasing sequence. We also define the \textit{tetris operation} $T : \FIN_k \to \FIN_{k-1}$ by stipulating that:
\begin{align*}
    T(x)(i) := \max\{x(i) - 1,0\}.
\end{align*}
For each such $\alpha$, the space of block sequences of length $\alpha$ is denoted by $\FIN_k^{[\alpha]}$ (so $\FIN_k^{[\infty]}$ is the space of all infinite block sequences, and $\FIN_k^{[<\infty]} := \bigcup_{n<\omega} \FIN_k^{[n]}$). Given a finite block sequence $A = (x_n)_{n < N}$, the \textit{partial semigroup} generated by $A$, denoted as $\c{A}$ or $\c{x_0,\dots,x_{N-1}}$, can be written in one of the following equivalent ways:
\begin{align*}
    \c{A} &:= \{T^{\lambda_0}(x_0) + \cdots + T^{\lambda_{N-1}}(x_{N-1}) : \lambda_i = 0 \text{ for some } i < N\}, \\
    &= \{x = T^{\lambda_0}(x_0) + \cdots + T^{\lambda_{N-1}}(x_{N-1}) : k \in \ran(x)\}.
\end{align*}
If $A = (x_n)_{n<\omega}$ is an infinite block sequence, then:
\begin{align*}
    \c{A} := \bigcup_{N<\omega} \c{x_0,\dots,x_{N-1}}.
\end{align*}
This allows us to define a partial order $\leq$ on $\FIN_k^{[\infty]}$, where $A \leq B$ iff $\c{A} subseteq \c{B}$.

The Ramsey theory of $\FIN_k$ was first explored by Gowers \cite{G92} in proving the oscillation stability of uniformly continuous functions $f : S_{c_0} \to \R$, where $S_{c_0}$ is the unit sphere in $c_0$. The current formulation of infinite block sequences is also due to Todor\v{c}evi\'{c} in \cite{T10}.

\begin{theorem}[Gowers $\FIN_k$ Theorem]
    For any $Y \subseteq \FIN_k$ and $A \in \FIN_k^{[\infty]}$, there exists some $B \leq A$ such that $\c{B} \subseteq Y$ or $\c{B} \subseteq Y^c$.
\end{theorem}

Call a subset of $Y \subseteq \FIN_k$ \emph{small} if there does not exist some $A \in \FIN_k^{[\infty]}$ such that $\c{A} \subseteq X$. This theorem implies that the set of small subsets of $\FIN_k$ forms an ideal - indeed, if $Y \cup Z \subseteq \FIN_k$ is not small, say $\c{A} \subseteq Y \cup Z$, then by the above theorem there exists some $B \leq A$ such that $\c{B} \subseteq Y$ or $\c{B} \subseteq Y^c$. In the latter case, $\c{B} \subseteq Z$. Consequently, it is reasonable to discuss almost disjoint families with respect to the ideal of small subsets of $\FIN_k$.

\begin{definition}
\hfill
    \begin{enumerate}
        \item Two subsets $Y,Z \subseteq \FIN_k$ are \emph{almost disjoint} (or simply \emph{ad}) if $Y \cap Z$ is small.

        \item Two infinite block sequences $A,B \in \FIN_k^{[\infty]}$ are \emph{almost disjoint} if $\c{A},\c{B}$ are almost disjoint.

        \item A family $\A \subseteq \FIN_k^{[\infty]}$ is an \emph{almost disjoint family} if for all $A \neq B$ in $\A$, $A$ and $B$ are almost disjoint. If $\A$ is maximal under inclusion, then $\A$ is called a \emph{maximal almost disjoint family}, or simply a \emph{mad family}.
    \end{enumerate}
\end{definition}
For $k = 1$, $A,B \in \FIN^{[\infty]}$ are almost disjoint iff $\c{A} \cap \c{B}$ are finite, so the uncountability of the cardinal $\frak{a}_\FIN$ is immediate. This particular cardinal was studied in \cite{BG17}, where it was shown that the inequality $\non(\M) \leq \frak{a}_\FIN$. Consequently, $\frak{a} < \frak{a}_\FIN$ in the random model. 

The uncountability of the cardinal $\frak{a}_{\FIN_k}$ is less apparent for $k \geq 2$. Unlike the case of $\FIN$, it is not true that if $A,B \in \FIN_k^{[\infty]}$ are almost disjoint, then $\c{A} \cap \c{B}$ is finite.

\begin{example}
\label{ex:intersection.small.but.infinite}
    Let $A = (x_n)_{n<\omega}$ and $B = (y_n)_{n<\omega}$ be defined as follows:
    \begin{enumerate}
        \item $\supp(x_0) = \supp(y_0) = \{0\}$, and $x_0(0) = y_0(0) = 2$.
        
        \item For all $n > 0$, $\supp(x_n) = \{2n-1\}$ and $x_n(2n-1) = 2$.
        
        \item For all $n > 0$, $\supp(y_n) = \{2n-1,2n\}$, $y_n(2n-1) = 2$ and $y_n(2n) = 1$.
    \end{enumerate}
    Observe that for all $n > 0$, $T(x_n) = T(y_n)$. Therefore, for any $n$, we have that:
    \begin{align*}
        z_n := x_0 + T(x_n) = y_0 + T(y_n) \in \c{A} \cap \c{B},
    \end{align*}
    and $\max(\supp(z)) = 2n-1$. Therefore, $\sup\{g(x) : x \in \c{A} \cap \c{B}\} = \infty$. We shall show that $\c{A} \cap \c{B/1} = \emptyset$, so $A$ and $B$ are almost disjoint by Lemma \ref{lem:ad.iff.tail.disjoint}.

    Suppose otherwise, and let $z \in \c{A} \cap \c{B/1}$. Then $z = \sum_{m \geq 1} T^{\mu_m}(y_m)$, and $\mu_m = 0$ for some $m > 0$. This implies that $2m \in \supp(y_m) \subseteq \supp(z)$. However, observe that $2m \notin \supp(A)$. Since $z \in \c{A}$, $\supp(z) \subseteq \supp(A)$ so $2m \notin \supp(z)$, a contradiction.
\end{example}

The main theorem of this paper is thus the following:

\begin{theorem}
\label{thm:a_FINk.is.uncountable}
    For all $k \geq 1$, $\frak{a}_{\FIN_k} \geq \aleph_1$.
\end{theorem}

\section{Proof of Theorem \ref{thm:a_FINk.is.uncountable}}
Fix some $k \geq 1$. Consider a function $f : \FIN_k \to \N$ satisfying the following two properties:
\begin{enumerate}
    \item For all $A,B \in \FIN_k^{[\infty]}$, $A$ and $B$ are almost disjoint iff $\sup\{f(x) : x \in \c{A} \cap \c{B}\} < \infty$.
    
    \item For all $A,B \in \FIN_k^{[\infty]}$ such that $A$ and $B$ are almost disjoint, and $a \in \FIN_k^{[<\infty]}$, there exists some $M < \omega$ such that $a < x$ for all $x \in \c{B/M}$, and:
    \begin{align*}
        \sup\{f(y) : y \in \c{a^\frown x} \cap \c{A}\} = \sup\{f(y) : y \in \c{a} \cap \c{A}\}.
    \end{align*}
    Here, if $B = (y_n)_{n<\omega}$ we denote $B/M := (y_n)_{n\geq M}$.
\end{enumerate}

\begin{lemma}
    If there exists a function $f : \FIN_k \to \N$ satisfying the above two properties, then every infinite mad family $\A \subseteq \FIN_k^{[\infty]}$ is uncountable.
\end{lemma}

\begin{proof}
    Let $\A := \{A_n : n < \omega\}$ be a countable almost disjoint family of infinite block sequences. We shall show that $\A$ is not maximal by constructing some $B = (y_n)_{n<\omega}$ that is almost disjoint from $A_n$ for all $n < \omega$ inductively as follows: If $a := (y_i)_{i<n}$ has been defined, then by property (2) of $f$. Let $M$ be large enough so that for all $x \in \c{A_n/M}$ and $i < n$:
    \begin{align*}
        \sup\{f(y) : y \in \c{a^\frown x} \cap \c{A_i}\} = \sup\{f(y) : y \in \c{a} \cap \c{A_i}\}.
    \end{align*}
    Let $y_n \in \c{A_n/M}$ be any block. Then, for all $n$:
    \begin{align*}
        \sup\{f(x) : x \in \c{A_n} \cap \c{B}\} &= \sup_{i<\omega} \sup\{f(x) : x \in \c{A_n} \cap \c{y_0,\dots,y_{i-1}}\} \\
        &= \sup\{f(x) : x \in \c{A_n} \cap \c{y_0,\dots,y_n}\} \\
        &< \infty,
    \end{align*}
    so by property (1) of $f$, $B$ is almost disjoint from $A_n$ for all $n < \omega$.
\end{proof}

Therefore, the following proposition would be sufficient to prove Theorem \ref{thm:a_FINk.is.uncountable}:

\begin{proposition}
\label{prop:FIN_k.iterable.valuation}
    The map $f : \FIN_k \to \N$ defined by:
    \begin{align*}
        f(x) := \max\{i \in \supp(x) : x(i) = k\}
    \end{align*}
    satisfies properties (1) and (2).
\end{proposition}

\begin{lemma}
\label{lem:ad.iff.tail.disjoint}
    Let $A,B \in \FIN_k^{[\infty]}$. Then $A$ and $B$ are almost disjoint iff there exists some $N$ such that $\c{A/N} \cap \c{B} = \emptyset$.
\end{lemma}

\begin{proof}
    We write $A = (x_n)_{n<\omega}$.

    \underline{$\impliedby$:} Suppose $\c{C} \subseteq \c{A} \cap \c{B}$ for some $C \in \FIN_k^{[\infty]}$. Write $C = (z_n)_{n<\omega}$. Then, for any $N$, there exists some $M$ such that $z_M > x_{n-1}$. Therefore, $z_n \in \c{A/N} \cap \c{B} \neq \emptyset$.

    \underline{$\implies$:} Suppose on the contrary that $\c{A/N} \cap \c{B} \neq \emptyset$ for all $N$. We define $z_n \in \FIN_k$ inductively as follows: Let $z_0 \in \c{A} \cap \c{B}$ be any block. If $z_n$ has been defined, then let $N$ be large enough so that $x_N > z_n$, and let $z_{n+1} \in \c{A/N} \cap \c{B}$. We have $z_{n+1} > z_n$, so if $C = (z_n)_{n<\omega} \in \FIN_k^{[\infty]}$, then $\c{C} \subseteq \c{A} \cap \c{B}$, as desired.
\end{proof}

\begin{lemma}
\label{lem:FIN_k.iterable.valuation.main.lemma}
    Let $A,B \in \FIN_k^{[\infty]}$, and suppose that:
    \begin{align*}
        \sup\{f(x) : x \in \c{A} \cap \c{B}\} = \infty.
    \end{align*}
    Then there exists some $N > 0$ such that:
    \begin{align*}
        \sup\{f(x) : x \in \c{A/N} \cap \c{B}\} = \infty.
    \end{align*}
\end{lemma}

\begin{proof}
    Let $A = (x_n)_{n<\omega},B = (y_n)_{n<\omega} \in \FIN_k^{[\infty]}$, and suppose that $\sup\{f(x) : x \in \c{A} \cap \c{B}\} = \infty$. For any $z \in \c{A} \cap \c{B}$, we may write:
    \begin{align*}
        z = \sum_{n<\omega} T^{\lambda_n}(x_n) = \sum_{m<\omega} T^{\mu_m}(y_m),
    \end{align*}
    where $\lambda_n,\mu_m \leq k$ for all $n,m$, $\lambda_n,\mu_m = k$ for all but finitely many $n,m$, and $\lambda_n < k$ for some $n$ and $\mu_m < k$ for some $m$. Let $G_z := (V_0 \sqcup V_1,E)$ be the bipartite graph where:
    \begin{enumerate}
        \item $V_0 := \{n < \omega : \lambda_n < k\}$.
        
        \item $V_1 := \{m < \omega : \mu_m < k\}$.
        
        \item $(n,m) \in E$ iff $\supp(T^{\lambda_n}(x_n)) \cap \supp(T^{\mu_m}(y_m)) \neq \emptyset$.
    \end{enumerate}

    We say that an element $x \in \c{A} \cap \c{B}$ is \emph{prime}\footnote{In an older version of this note, this property was named \emph{intertwined}.} if $G_z$ is connected.

    \begin{claim}
        There exists a prime $p \in \c{A} \cap \c{B}$.
    \end{claim}

    \begin{midproof}
        Let $N$ is the least number such that $\c{x_0,\dots,x_{N-1}} \cap \c{B} \neq \emptyset$, and let $z \in \c{x_0,\dots,x_{N-1}} \cap \c{B}$ be any element. Write $z = \sum_{n<N} T^{\lambda_n}(x_n) = \sum_{m<\omega} T^{\mu_m}(y_m)$. Note that by the minimality of $N$, $\lambda_{N-1} < k$. Let $M$ be the largest integer such that $\mu_{M-1} < k$, so we have that $z \in \c{x_0,\dots,x_{N-1}} \cap \c{y_0,\dots,y_{M-1}}$. Let $N' < N$ be the largest integer such that $\min(\supp(T^{\lambda_{N'}}(x_{N'}))) = \min(\supp(T^{\mu_{M'}}(y_{M'})))$ for some $M'$. Note that this is equivalent to saying that $N' < N$ is the largest integer such that $\max(\supp(T^{\lambda_{N'-1}}(x_{N'-1}))) = \max(\supp(T^{\mu_{M'-1}}(y_{M'-1})))$. Consider writing $z$ in the following manner:
        \begin{alignat*}{3}
            z &= \underbrace{\sum_{n < N'} T^{\lambda_n}(x_n)}_{z_0^0} &&\; + &&\underbrace{\sum_{n \geq N'} T^{\lambda_n}(x_n)}_{z_1^0}, \\
            z &= \underbrace{\sum_{m < M'} T^{\mu_m}(y_m)}_{z_0^1} &&\; + &&\underbrace{\sum_{m \geq M'} T^{\mu_m}(y_m)}_{z_1^1}.
        \end{alignat*}
        We have the following:
        \begin{enumerate}
            \item $z = \supp(z_0^0) \cup \supp(z_1^0) = \supp(z_0^1) \cup \supp(z_1^1)$.
            
            \item $\max(\supp(z_0^0)) < \min(\supp(z_1^0))$.
            
            \item $\max(\supp(z_0^1)) < \min(\supp(z_1^1))$.
        \end{enumerate}
        Since $\min(\supp(z_1^0)) = \min(\supp(z_1^1))$, we have that $\max(\supp(z_0^0)) = \max(\supp(z_0^1))$. Therefore, $z_0^0 = z_0^1$ and $z_1^0 = z_1^1$. We denote them as $z_0$ and $z_1$ respectively. Observe that:
        \begin{align*}
            \ran(z) = \ran(z_0) \cup \ran(z_1).
        \end{align*}
        If $k \notin \ran(z_1)$, then $k \in \ran(z_0)$, so $z_0 \in \c{x_0,\dots,x_{N'-1}} \cap \c{B}$. This contradicts the minimality of $N$ as $N' < N$, so $k \in \ran(z_1)$. Therefore, $z_1 \in \c{x_0,\dots,x_{N-1}} \cap \c{B}$. We shall show that $G_{z_1}$ is connected, so $z_1$ is prime.

        Write $G_{z_1} = (V_0' \sqcup V_1',E')$, where $V_0' \subseteq V_0$ and $V_1' \subseteq V_1$. Observe that:
            \begin{align*}
                V_0 &= \{n \in V_0 : n \geq N'\}, \\
                V_1 &= \{m \in V_1 : m \geq M'\}.
            \end{align*}
            Write $V_0 = \{n_0,\dots,n_{s-1}\}$ and $V_1 = \{m_0,\dots,m_{t-1}\}$, where $N' = n_0 < \cdots < n_{s-1} = N - 1$ and $M' = m_0 < \cdots < m_{t-1} = M - 1$. For $u < s$ and $v < t$, let $H_{u,v}$ be the induced subgraph of $G_{z_1}$ with vertices $\{n_0,\dots,n_u\} \sqcup \{m_0,\dots,m_v\}$. We shall show by induction that $H_{u,v}$ is connected for all $u,v$. For the base case, $H_{0,0}$ is connected as $\min(\supp(T^{\lambda_{N'}}(x_{N'}))) = \min(\supp(T^{\mu_{M'}}(y_{M'})))$. Suppose that $H_{u,v}$ is connected. If $u = s - 1$ (and so $v = t - 1$), then we're done, so assume otherwise. By the maximality of $N'$, we have that $\max(\supp(T^{\lambda_{n_u}}(x_{n_u}))) \neq \max(\supp(T^{\mu_{m_v}}(y_{m_v})))$. Assume WLOG that $\max(\supp(T^{\lambda_{n_u}}(x_{n_u}))) < \max(\supp(T^{\mu_{m_v}}(y_{m_v})))$, and the other case is similar. This implies that:
            \begin{align*}
                \min(\supp(T^{\lambda_{n_{u+1}}}(x_{n_{u+1}}))) \in \supp(T^{\mu_{m_v}}(y_{m_v})),
            \end{align*}
            so $(n_{u+1},m_v) \in E$. Since $H_{u,v}$ is connected and $n_{u+1}$ is connected to $H_{u,v}$, $H_{u+1,v}$ is also connected, completing the induction.
    \end{midproof}

    Before we proceed, we first introduce a binary operation $\star$ on $\FIN_k$, where for any $x,y \in \FIN_k$ and $i < \omega$:
    \begin{align*}
        (x \star y)(i) := \max\{x(i),y(i)\}.
    \end{align*}
    We have that $\supp(x \star y) = \supp(x) \cup \supp(y)$, which is finite, and if $x(i) = k$ or $y(i) = k$, then $(x \star y)(i) = k$. Therefore, $\FIN_k$ is closed under $\star$. It's easy to verify that if $z,w \in \c{A}$, where $z = \sum_{n<\omega} T^{\lambda_n}(x_n)$ and $w = \sum_{n<\omega} T^{\lambda_n'}(x_n)$, then:
    \begin{align*}
        z \star w = \sum_{n<\omega} T^{\min\{\lambda_n,\lambda_n'\}}(x_n).
    \end{align*}
    In particular, $\c{A}$ (and similarly $\c{B}$) is closed under $\star$. It's also easy to see that for any $z,w \in \c{A} \cap \c{B}$:
    \begin{align*}
        f(z \star w) = \max\{f(z),f(w)\}.
    \end{align*}
    
    We now fix some prime $p \in \c{A} \cap \c{B}$, and write:
    \begin{align*}
        p = \sum_{u<s} T^{\lambda_{n_u}}(x_{n_u}) = \sum_{m<\omega} T^{\mu_{m_v}}(y_{m_v}),
    \end{align*}
    where $n_0 < \cdots < n_{s-1}$, $m_0 < \cdots < m_{t-1}$ and $\lambda_{n_u},\mu_{m_v} < k$ for all $u < s$ and $v < t$.

    \begin{claim}
        For any $z \in \c{A} \cap \c{B}$, if $i \in \supp(z)$ and $\min(\supp(p)) \leq i \leq \max(\supp(p))$, then $(p \star z)(i) = p(i)$. 
    \end{claim}

    \begin{midproof}
        Let $z \in \c{A} \cap \c{B}$, and we write:
        \begin{align*}
            z = \sum_{n<\omega} T^{\lambda_n'}(x_n) = T^{\mu_m'}(y_m).
        \end{align*}
        Let $i \in \supp(z)$, so $i \in \supp(T^{\lambda_n'}(x_n)) \cap \supp(T^{\mu_m'}(y_m))$ for some $n,m$. We claim that $n = n_u$ for some $u < s$ or $m = m_v$ for some $v < s$. Otherwise, since $\min(\supp(p)) \leq i \leq \max(\supp(p))$, there exists some $u < s - 1$ and $v < t - 1$ such that $n_u < n < n_{u+1}$ and $m_v < m < m_{v+1}$. This implies that:
        \begin{gather*}
            \max(\supp(T^{\lambda_{n_u}}(x_{n_u}))) < i < \min(\supp(T^{\lambda_{n_{u+1}}}(x_{n_{u+1}}))), \\
            \max(\supp(T^{\mu_{m_v}}(y_{m_v}))) < i < \min(\supp(T^{\mu_{m_{v+1}}}(y_{m_{v+1}}))).
        \end{gather*}
        Since $G_p$ is connected, either there exists some $u' \leq u$ and $v' \geq v+1$ such that $\supp(T^{\lambda_{n_u'}}(x_{n_{u'}})) \cap \supp(T^{\mu_{m_{v'}}}(y_{m_{v'}})) \neq \emptyset$, or there exists some $u' \geq u+1$ and $v' \leq v$ such that $\supp(T^{\lambda_{n_{u'}}}(x_{n_{u'}})) \cap \supp(T^{\mu_{m_{v'}}}(y_{m_{v'}})) \neq \emptyset$. In the first case, let $j \in \supp(T^{\lambda_{n_u'}}(x_{n_{u'}})) \cap \supp(T^{\mu_{m_{v'}}}(y_{m_{v'}}))$. This is a contradiction, as $j \in \supp(T^{\lambda_{n_u'}}(x_{n_{u'}}))$ implies that $j < i$, but $j \in \supp(T^{\mu_{m_{v'}}}(y_{m_{v'}}))$ implies that $j > i$. We may obtain a similar contradiction in the second case. 

        Assume WLOG that $n = n_u$ for some $u < s$. As discussed prior to the lemma, we have that:
        \begin{align*}
            p \star z = \sum_{u<s} T^{\min\{\lambda_{n_u},\lambda_{n_u}'\}}(x_{n_u}) = \sum_{v<t} T^{\min\{\mu_{m_v},\mu_{m_v}'\}}(y_{m_v}).
        \end{align*}
        Suppose for a contradiction that $(p \star z)(i) > p(i)$, i.e. $\lambda_{n_u}' < \lambda_{n_u}$. We claim that for all $u' < s$, $\lambda_{n_{u'}}' < \lambda_{n_{u'}}$. This leads to our desired contradiction, as there exists some $u' < s$ be such that $\lambda_{n_{u'}} = 0$.

        Since $G_p$ is connected, there is a path between $n_{u'}$ and $n_u$ of length $l$. Note that $l$ is even, as $G_p$ is bipartite. We derive a contradiction by induction on $l$. If $l = 0$, then $n_{u'} = n_u$, so $\lambda_{n_u}' < \lambda_{n_u}$. Now suppose that $l = l' + 2$. Let $v' < t$ and $u'' < s$ such that there is a path of length $l$ between $n_{u'}$ and $n_u$, and:
        \begin{align*}
            \supp(T^{\lambda_{u'}}(x_{u'})) \cap \supp(T^{\mu_{v'}}(y_{v'})) &\neq \emptyset, \\
            \supp(T^{\lambda_{u''}}(x_{u''})) \cap \supp(T^{\mu_{v'}}(y_{v'})) &\neq \emptyset. 
        \end{align*}
        By the induction hypothesis, $\lambda_{u''}' < \lambda_{u''}$. Let $j_0 \in \supp(T^{\lambda_{u''}}(x_{u''})) \cap \supp(T^{\mu_{v'}}(y_{v'}))$. Then:
        \begin{align*}
            T^{\min\{\mu_{v'},\mu_{v'}'\}}(y_{v'})(j_0) &= T^{\min\{\lambda_{u''},\lambda_{u''}'\}}(x_{u''})(j_0) \\
            &> T^{\lambda_{u''}}(x_{u''})(j_0) \\
            &= T^{\mu_{v'}}(y_{v'})(j_0),
        \end{align*}
        so $\mu_{v'}' < \mu_{v'}$. Now let $j_1 \in \supp(T^{\lambda_{u'}}(x_{u'})) \cap \supp(T^{\mu_{v'}}(y_{v'}))$. Similarly, we have that:
        \begin{align*}
            T^{\min\{\lambda_{u'},\lambda_{u'}'\}}(x_{u''})(j_1) &= T^{\min\{\mu_{v'},\mu_{v'}'\}}(y_{v'})(j_1) \\
            &> T^{\mu_{v'}}(y_{v'})(j_1) \\
            &= T^{\lambda_{u'}}(x_{u'})(j_1),
        \end{align*}
        so $\lambda_{u'}' < \lambda_{u'}$. This completes the induction.
    \end{midproof}

    \begin{claim}
        For any $z \in \c{A} \cap \c{B}$, there exist some $w_0,w_1 \in \c{A} \cap \c{B}$ and some $\alpha_0,\alpha_1 \leq k$ such that:
        \begin{align*}
            p \star z = T^{\alpha_0}(w_0) + p + T^{\alpha_1}(w_1),
        \end{align*}
        and $w_0 < p < w_1$.
    \end{claim}

    \begin{midproof}
        We write:
        \begin{align*}
            p \star z = \sum_{n<\omega} T^{\lambda_n''}(x_n) = \sum_{m<\omega} T^{\mu_m''}(y_m).
        \end{align*}
        By the previous claim, we have that for all $n_0 \leq n \leq n_{s-1}$, $\lambda_n'' = \lambda_n$. Therefore, we may write:
        \begin{align*}
            p \star z &= \sum_{n<n_0} T^{\lambda_n''}(x_n) + p + \sum_{n>n_{s-1}} T^{\lambda_n''}(x_n), \\
            &= \sum_{m<m_0} T^{\mu_m''}(y_m) + p + \sum_{m>m_{t-1}} T^{\mu_m''}(y_m).
        \end{align*}
        This implies that:
        \begin{align*}
            \sum_{n<n_0} T^{\lambda_n''}(x_n) &= \sum_{m<m_0} T^{\mu_m''}(y_m), \\
            \sum_{n>n_{s-1}} T^{\lambda_n''}(x_n) &= \sum_{m>m_{t-1}} T^{\mu_m''}(y_m). 
        \end{align*}
        Let $\alpha_0 := \min\{\lambda_n'' : n < n_0\}$ and $\alpha_1 := \min\{\lambda_n'' : n > n_{s-1}\}$. Then:
        \begin{align*}
            \sum_{n<n_0} T^{\lambda_n''-\alpha_0}(x_n) &= \sum_{m<m_0} T^{\mu_m''-\alpha_0}(y_m) =: w_0, \\
            \sum_{n>n_{s-1}} T^{\lambda_n''-\alpha_1}(x_n) &= \sum_{m>m_{t-1}} T^{\mu_m''-\alpha_1}(y_m) =: w_1. 
        \end{align*}
        Then $w_0,w_1 \in \c{A} \cap \c{B}$ and $p \star z = T^{\alpha_0}(w_0) + p + T^{\alpha_1}(w_1)$, as desired.
    \end{midproof}

    We may now complete the proof of the lemma. Let $N$ be the smallest positive integer such that $\c{x_0,\dots,x_{N-1}} \cap \c{B}$. Let $K$ be an arbitrary integer. Since $\sup\{f(x) : x \in \c{A} \cap \c{B}\} = \infty$, there exists some $z \in \c{A} \cap \c{B}$ such that $f(p) < f(z)$ and $K \leq f(z)$. Therefore, $f(p \star z) = f(z)$. By the final claim, we may write $p \star z = T^{\alpha_0}(w_0) + p + T^{\alpha_1}(w_1)$ for some $w_0,w_1 \in \c{A} \cap \c{B}$ and some $\alpha_0,\alpha_1$, where $w_0 < p < w_1$. If $\alpha_1 > 0$, then $k \notin \ran(T^{\alpha_1}(w_1))$, so $f(p \star z) \leq f(p)$, a contradiction. This implies that $\alpha_1 = 0$, so $p < w_1$, and therefore $w_1 \in \c{A/N} \cap \c{B}$. Hence:
    \begin{align*}
        \sup\{f(x) : x \in \c{A/N} \cap \c{B}\} \geq f(w_1) = f(z) \geq K.
    \end{align*}
    Since $K$ may be chosen to be arbitrarily large, $\sup\{f(x) : x \in \c{A/N} \cap \c{B}\} = \infty$.
\end{proof}

\begin{proof}[Proof of Proposition \ref{prop:FIN_k.iterable.valuation}]
    We first show that $f$ satisfies property (1). Suppose that $A$ is compatible with $B$. Let $C = (x_n)_{n<\omega}$ be such that $\c{C} \subseteq \c{A} \cap \c{B}$. Then $(f(x_n))_{n<\omega}$ is an unbounded sequence of natural numbers, so $\sup\{f(x) : x \in \c{A} \cap \c{B}\} = \infty$. 

    Conversely, suppose that $\sup\{f(x) : x \in \c{A} \cap \c{B}\} = \infty$. By Lemma \ref{lem:ad.iff.tail.disjoint}, it suffices to show that there exists a strictly increasing sequence $(N_i)_{i<\omega}$ of natural numbers such that $\c{A/N_i} \cap \c{B} \neq \emptyset$ for all $i$. Let $N_0 := 0$. If $\sup\{f(x) : x \in \c{A/N_i} \cap \c{B}\} = \infty$, by Lemma \ref{lem:FIN_k.iterable.valuation.main.lemma} there exists some $N_{i+1} > N_i$ such that $\sup\{f(x) : x \in \c{A/N_{i+1}} \cap \c{B}\} = \infty$. In particular, $\c{A/N_i} \cap \c{B} \neq \emptyset$ for all $i$, as desired.

    We show that $f$ satisfies property (2). Let $A = (x_n)_{n<\omega}$ and $B = (y_n)_{n<\omega}$ be almost disjoint, and let $a \in \FIN_k^{[<\infty]}$. Let $N$ be large enough so that $a < A/N$. By Lemma \ref{lem:ad.iff.tail.disjoint}, let $M$ be large enough such that $a < B/M$, $x_{N-1} < B/M$ and $\c{A} \cap \c{B/M} = \emptyset$. We claim that this $M$ works. Let $x \in \c{B/M}$, and let $y \in \c{a^\frown x} \cap \c{A}$. Since $y \in \c{a^\frown x}$, $y = T^\lambda(z) + T^\alpha(x)$ for some $z \in \c{a}$, and that $\lambda = 0$ or $\alpha = 0$. On the other hand, we write:
    \begin{align*}
        y = \sum_{n<\omega} T^{\lambda_n}(x_n) = \sum_{n < N} T^{\lambda_n}(x_n) + \sum_{n \geq N} T^{\lambda_n}(x_n).
    \end{align*}
    Since $\sum_{n < N} T^{\lambda_n}(x_n) < x$ and $a < \sum_{n \geq N} T^{\lambda_n}(x_n)$, we have that:
    \begin{align*}
        T^\lambda(w) &= \sum_{n < N} T^{\lambda_n}(x_n), \\
        T^\alpha(x) &= \sum_{n \geq N} T^{\lambda_n}(x_n).
    \end{align*}
    If $\alpha = 0$, then $x \in \c{A} \cap \c{B/M}$, a contradiction. Therefore, $y = z + T^\alpha(x)$ for some $\alpha > 0$, so $f(y) = f(z)$. Since $z \in \c{a} \cap \c{B}$, we have that $f(z) \leq \sup\{f(y) : y \in \c{a} \cap \c{A}\}$, as desired.
\end{proof}

\printbibliography[heading=bibintoc,title={References}]

\end{document}